\documentclass[12pt]{article}
\usepackage[leqno]{amsmath}
\usepackage[margin=1.1in]{geometry}
\usepackage{tikz}
\usetikzlibrary{arrows.meta,calc}
\usepackage{amssymb,amsthm,mathtools}
\usepackage{enumitem}
\usepackage[hidelinks]{hyperref}
\usepackage{microtype}
\usepackage{comment}
\newtheorem{theorem}{Theorem}
\newtheorem{proposition}[theorem]{Proposition}
\newtheorem{lemma}[theorem]{Lemma}
\newtheorem{corollary}[theorem]{Corollary}
\newtheorem{question}[theorem]{Question}
\theoremstyle{remark}
\newtheorem{remark}[theorem]{Remark}

\newcommand{\R}{\mathbb R}

\newcommand{\Kn}{\mathcal{K}^n}

\newcommand{\pol}[1]{#1^{\circ}}
\newcommand{\conv}{\operatorname{conv}}

\newcommand{\ip}[2]{\left\langle #1,#2\right\rangle}

\newcommand{\intt}{\operatorname{int}}

\title{Self-dual sets up to a translation: a negative answer to Milman's question}
\author{Alex Segal}

\date{}

\begin{document}

\maketitle

\begin{abstract}
Let \(K,T\subset\mathbb R^n\) be convex bodies with the origin in their
interiors. If \(K=T^\circ\), then the identity
\(
K+T=K^\circ+T^\circ
\)
holds trivially. V.~Milman asked whether these trivial solutions are the only
ones, that is, whether
\(
K+T=K^\circ+T^\circ
\)
forces \(K=T^\circ\). We show that the answer is positive for polytopes but
negative for general convex bodies.
\end{abstract}

\section{Introduction}
Let $\Kn$ denote the family of compact convex bodies in $\R^n$ that contain the
origin in their interior. For $K \in \Kn$ the \emph{polar} (or \emph{dual}) body is
\begin{equation}\label{eq:polar}
  \pol{K} \;=\; \bigl\{\, y \in \R^n : \ip{x}{y} \le 1 \text{ for all } x \in K \,\bigr\},
\end{equation}
which again belongs to $\Kn$. Polarity is the fundamental order-reversing
involution of convex geometry: it satisfies the bipolar relation
$\pol{(\pol{K})} = K$, reverses order, and interchanges the support function
$h_K(u) = \sup_{x \in K} \ip{x}{u}$ with the Minkowski gauge functional $p_K$ through the
identity $h_{\pol{K}} = p_K$ (see \cite{schneider} for details). 
 
The structural rigidity of the polarity map has been the subject of extensive
study. A guiding principle is that polarity is essentially the
\emph{only} natural transformation of $\Kn$ with properties described above. Artstein-Avidan and Milman showed that an order-reversing involution on the class of closed convex sets containing the origin must be polarity up to a linear change of variables~\cite{AM-duality}.
Similar characterizations for other classes were obtained by B\"or\"oczky and
Schneider~\cite{BS-duality}. Another characterizing property of duality is that it interchanges intersection and convex hull. For results on the subject see e.g \cite{gruber, BS-duality, slomka}.

Similarly, it was shown that the Legendre transform is
 the unique order-reversing involution of convex
functions~\cite{AM-legendre}. These results show that
duality is characterized by maps with minimal properties. However, no such characterization seems to exist through solutions of equations involving convex sets.
 
Such an identity is the subject of the present note. Every convex body and its polar
satisfy $K + \pol{K} = \pol{K} + \pol{(\pol{K})}$, so each dual pair
$(K, \pol{K})$ solves
\begin{equation}\label{eq:main}
  K + T \;=\; \pol{K} + \pol{T}.
\end{equation}
Since both \eqref{eq:main} and the relation $K = \pol{T}$ are symmetric in $K$ and
$T$, it is natural to ask whether these dual pairs are the only solutions. This
question was raised by V.~Milman. We prove that the answer is positive in the class of polytopes, but negative for general convex bodies.
The existence of bodies that are self-dual up to an affine transformation is of independent interest. For instance, Jensen \cite{jensen} recently investigated the structure of self-dual polytopes  up to an orthogonal transformation. In a similar spirit, our counterexample is formed by constructing two distinct convex bodies that are self-dual up to a common translation. Namely, we prove the following:

\begin{theorem} \label{thm-main-polyopes}
Let $n \geq 2$. Assume that $K, T \in \Kn$ are polytopes, such that
\[
K + T = K^\circ + T^\circ.
\]
Then, $K=T^\circ$.
\end{theorem}

\begin{theorem} \label{thm-main1}
Let $n \geq 2$ and let $0\neq s \in \R^n$. Then, there exist two different convex bodies $K_1, K_2 \in \Kn$ such that 
\[
K_i^\circ = K_i - s, \qquad i=1,2.
\]
\end{theorem}
Theorem \ref{thm-main1} provides an immediate counterexample to the conjecture by taking $K = K_1, \, T=K_2^\circ$:
\[
K^\circ + T^\circ = K_1 - s + K_2 = K_1 + K_2^\circ = K + T.
\]
To prove Theorem \ref{thm-main1}, we will build it in the two-dimensional plane first, by creating a suitable infinite polygon. To generalize it to higher dimension, we will consider its body of revolution. \\
In addition, we discuss a question closely related to (\ref{eq:main}), raised by V. Milman as well: 
\begin{question}
    Assume $K,T \in \Kn$, such that $K + K^\circ = T + T^\circ$. Does it follow that $K=T$ or $K=T^\circ$?
\end{question}
We show that the answer is negative as well.
\begin{theorem} \label{thm-main2}
Let $n \geq 2$. Then, there exist two different convex bodies $K, T \in \Kn$ such that 
\[
K + K^\circ = T + T^\circ,
\]
and $K\neq T^\circ, \,\, K \neq T$.
\end{theorem}
The latter is constructed, using sets related to the proof of Theorem \ref{thm-main1}. 

\section{Polytopes}
Notice that the main conjecture is equivalent to showing that if $K + T^\circ = K^\circ + T$ then $K=T$. In terms of support functions, the equation is equivalent to 
\[
h_K + h_{T^\circ} = h_T + h_{K^\circ},
\]
or 
\begin{equation} \label{eq:support_radial_form}
h_K + \frac{1}{\rho_T} = h_T + \frac{1}{\rho_K},
\end{equation}
on the unit sphere $S^{n-1}$, where $\rho_K$ is the radial function of $K$.
Given a convex set $K \in \Kn$, define $F_K:S^{n-1} \to \R$ by
\[
F_K = h_K - \frac{1}{\rho_K}.
\]
\begin{remark}
Given $K\in \Kn$ and $u\in S^{n-1}$ we will say that $x\in \partial K$ is a support point of $K$ with normal $u$ if there exists a supporting hyperplane of $K$ with normal $u$, containing $x$. Equivalently, $h_K(u) = \ip{x}{u}.$
\end{remark} 
Hence, to prove Theorem \ref{thm-main-polyopes}, we must show that $F_K$ is injective on the class of polytopes. To this end, we show the following claims.

\begin{lemma}\label{lem:ascent}
Let $K\in \Kn$ and let $x\in \partial K$ be a support point of $K$ with outer normal $u\in S^{n-1}$.
Then
\[
        F_K(v)\ge F_K(u),
\]
for \[
        v = \frac{x}{\|x\|}.
\]
In addition, if \(v\ne u\), then the inequality is strict.
\end{lemma}

\begin{proof}
Denote $c=\ip{u}{v}$.
By definition of \(x\), we have
\[
        h_K(u)=\ip{x}{u}=\rho_K(v)\cdot \ip{v}{u}=c\rho_K(v).
\]
The convex set $K$ contains $0$ in its interior, so $h_K(u) > 0$ which yields $0 < c \leq 1$ and $\rho_K(v) \geq h_K(u)$. Considering the above, we conclude
\begin{align*}
        F_K(v)-F_K(u)
        &=\left(h_K(v)-\frac1{\rho_K(v)}\right)
          -\left(h_K(u)-\frac1{\rho_K(u)}\right) \\
        &\ge \left(\rho_K(v)-\frac1{\rho_K(v)}\right)-\left(h_K(u)-\frac1{h_K(u)}\right) \\
        &= \left(\rho_K(v)-h_K(u)\right)+\left(\frac{\rho_K(v)-h_K(u)}{\rho_K(v)h_K(u)}\right) \geq 0.
\end{align*}
Notice that the inequality is strict, unless $\rho_K(v)=h_K(u)$, which is equivalent to $u=v$. This completes the proof.
\end{proof}

\begin{lemma} \label{lem:sign-flip}
    Let $K,T \in \Kn$ such that $F_K = F_T$. Let $u \in S^{n-1}$ such that 
    \[
    \Delta(u):= h_K(u) - h_T(u) >  0.
    \]
    Then, for every support point $x\in \partial K$ with outer normal $u$, the direction 
    \[
    v = \frac{x}{\|x\|}
    \]
    satisfies $\Delta(v) < 0$ and $F_K(v) > F_K(u)$.
    Similarly, if $\Delta(u) < 0$ then, for every support point $x\in \partial T$ with normal $u$, the direction 
    \[
    v = \frac{x}{\|x\|}
    \]
    satisfies $\Delta(v) > 0$ and $F_T(v) > F_T(u).$
\end{lemma}
\begin{proof}
Assume that \(\Delta(u)>0\).  Then
\[
        h_K(u)>h_T(u).
\]
Let \(x\in \partial K\) be a support point of \(K\) with outer normal \(u\), so that
\[
        \ip{x}{u}=h_K(u).
\]
Notice that $x\not\in T$, since every point \(y\in T\) satisfies
\[
        \ip{y}{u}\le h_T(u)<h_K(u)=\ip{x}{u}.
\]
Since \(x\in\partial K\) and \(0\in\intt K\), 
\[
        x=\rho_K(v)v.
\]
We know that \(x\notin T\), so $\rho_T(v)<\rho_K(v)$.
Consequently
\[
        \Delta(v)=\frac1{\rho_K(v)}-\frac1{\rho_T(v)}<0.
\]
This proves the first assertion. Now, by Lemma \ref{lem:ascent} We know that $F_K(v) \geq F_K(u)$. Clearly, $u\neq v$, since otherwise we would have $\Delta(u)=\Delta(v) > 0$. Hence $u\neq v$ and $F_K(v) > F_K(u).$ The case of $\Delta(u) < 0$ is similar.
\end{proof}

Now we turn to the proof of Theorem \ref{thm-main-polyopes}. 
\subsection{Proof of Theorem \ref{thm-main-polyopes}}
Let $K,T \in \Kn$ be two polytopes such that $K + T^\circ = K^\circ + T.$ Define \(\mathcal V_K\subset S^{n-1}\) to be the finite set of radial directions of
vertices of \(K\):
\[
        \mathcal V_K=
        \left\{\frac{x}{\|x\|}:x\text{ is a vertex of }K\right\}.
\]
Analogously, define \(\mathcal V_T\) to be the radial directions of vertices of $T$, and the union
\[
        \mathcal V=\mathcal V_K\cup\mathcal V_T,
\]
which is clearly a finite subset of \(S^{n-1}\).

Assume that \(\Delta:=h_K - h_T \not\equiv0\) and choose \(u_0\in S^{n-1}\) with \(\Delta(u_0)\ne0\). If $\Delta(u_0) > 0$ consider the exposed face
\[
        E_K(u_0):=\{x\in K:\ip{x}{u_0}=h_K(u_0)\}
\]

 Since \(K\) is a polytope, this face contains at least one vertex.  Choose a
vertex \(x_1\in E_K(u_0).\) 
Similarly, If \(\Delta(u_0)<0\), choose a vertex \(x_1\) of \(T\) in the exposed face
 \[
 E_T(u_0)=\{x\in T:\langle x,u_0\rangle=h_T(u_0)\}.
 \]
Define
\[
        u_1=\frac{x_1}{\|x_1\|}\in\mathcal V.
\]
By Lemma~\ref{lem:sign-flip}, and using \(F_K=F_T\) in the case
  \(\Delta(u_0)<0\) we get
\[
        \Delta(u_1) \neq 0,
        \qquad
        F_K(u_1)>F_K(u_0).
\]

Now repeat the process.  Suppose \(u_j\in\mathcal V\) has already been
constructed and \(\Delta(u_j)\ne0\).  If \(\Delta(u_j)>0\), choose a vertex of
\(K\) in the exposed face with normal \(u_j\).  If \(\Delta(u_j)<0\), choose a
vertex of \(T\) in the exposed face with normal \(u_j\).  Let \(u_{j+1}\) be
the radial direction of the chosen vertex.  Lemma~\ref{lem:sign-flip} gives
\[
        \Delta(u_{j+1}) \neq 0, \qquad F_K(u_{j+1})>F_K(u_j).
\]
Notice that in the last inequality we've used the fact that $F_K = F_T$.
Thus we obtain an infinite sequence
\[
        u_1,u_2,u_3,\dots\in\mathcal V
\]
such that
\[
        F_K(u_1)<F_K(u_2)<F_K(u_3)<\cdots.
\]
But \(\mathcal V\) is finite, so there exist $k \neq m$ such that $u_k = u_m$, which is a contradiction to the monotonicity of $\{F(u_j)\}$. 
Hence, we conclude that $\Delta \equiv 0$ and $K=T$.
\qed

\section{Counterexample for the general setting}
Now we turn to the proof of Theorem \ref{thm-main1}. We first construct a two dimensional example, and then we show how it generalizes to $\R^n$.

\subsection{Two dimensional construction}
Ideally, we would like to build a convex polygon $K$ such that its polar is $K-s$ for some shift $s$.  
Consider a vertex $p \in K$. In our desired construction, $p-s$ must be a vertex in $K^\circ$, so $\langle p-s, x\rangle = 1$ will be the supporting line of $K$.
Thus, it is natural to require that $\langle p-s, p\rangle = 1$.
The latter implies that each vertex of $K$ must lie on the circle 
\begin{equation}\label{eq:Gamma-s}
        \Gamma_s
        :=\left\{p\in\R^2:\ip{p-s}{p}=1\right\}
        =\left\{p:\left|p-\frac{s}{2}\right|^2=1+\frac{|s|^2}{4}\right\}.
\end{equation}

Without loss of generality, assume that $s = (m, 0)$ for some $m > 0$ and denote the radius of $\Gamma_s$ by 
\[
R = \sqrt{1+\frac{m^2}{4}}.
\]
Since all vertices in our construction lie on a circle, we may parametrize our problem. We consider the rational parametrization of $\Gamma_s$:
\[
p(t) = \left(\frac{m}{2} + R\frac{1-t^2}{1+t^2}, R\frac{2t}{1+t^2}\right), \qquad t \in \mathbb{R} \cup \{\infty\}.
\]
If $p(t)$ is a vertex in $K$, then $q(t)=p(t)-s$ is a vertex in $K^\circ$. We would like to determine the consecutive vertex of $p(u)$ of $K$ such that $q(t)$ is the vertex that defines the supporting line of $K$ on the edge $[p(t), p(u)]$ (see Figure \ref{fig:q_t_construction}).

\begin{figure}[h]
\centering
\begin{tikzpicture}[scale=2.25,>=Latex]
  \draw[thin] (1,0) circle ({sqrt(2)});

  % Same finite truncation.
  \draw[thick]
    (-0.414214,0.000000) --
    (-0.414213,-0.000421) --
    (-0.414211,-0.002451) --
    (-0.414141,-0.014285) --
    (-0.411765,-0.083189) --
    (-0.333333,-0.471405) --
    (1.000000,-1.414214) --
    (2.333333,-0.471405) --
    (2.411765,-0.083189) --
    (2.414141,-0.014285) --
    (2.414211,-0.002451) --
    (2.414213,-0.000421) --
    (2.414214,0.000000) --
    (2.414213,0.000421) --
    (2.414211,0.002451) --
    (2.414141,0.014285) --
    (2.411765,0.083189) --
    (2.333333,0.471405) --
    (1.000000,1.414214) --
    (-0.333333,0.471405) --
    (-0.411765,0.083189) --
    (-0.414141,0.014285) --
    (-0.414211,0.002451) --
    (-0.414213,0.000421) -- cycle;

  \fill (-0.414214,0.000000) circle (0.018);
  \fill (-0.414213,-0.000421) circle (0.018);
  \fill (-0.414211,-0.002451) circle (0.018);
  \fill (-0.414141,-0.014285) circle (0.018);
  \fill (-0.411765,-0.083189) circle (0.018);
  \fill (-0.333333,-0.471405) circle (0.018);
  \fill (1.000000,-1.414214) circle (0.018);
  \fill (2.333333,-0.471405) circle (0.018);
  \fill (2.411765,-0.083189) circle (0.018);
  \fill (2.414141,-0.014285) circle (0.018);
  \fill (2.414211,-0.002451) circle (0.018);
  \fill (2.414213,-0.000421) circle (0.018);
  \fill (2.414214,0.000000) circle (0.018);
  \fill (2.414213,0.000421) circle (0.018);
  \fill (2.414211,0.002451) circle (0.018);
  \fill (2.414141,0.014285) circle (0.018);
  \fill (2.411765,0.083189) circle (0.018);
  \fill (2.333333,0.471405) circle (0.018);
  \fill (1.000000,1.414214) circle (0.018);
  \fill (-0.333333,0.471405) circle (0.018);
  \fill (-0.411765,0.083189) circle (0.018);
  \fill (-0.414141,0.014285) circle (0.018);
  \fill (-0.414211,0.002451) circle (0.018);
  \fill (-0.414213,0.000421) circle (0.018);

  % Support line <q(1),x>=1.
  \draw[dashed,very thick] (2.636515,2.571405) -- (-3.303182,-1.628595);
  \node[right] at (-3.000000,-0.471405) {$\langle q(t),x\rangle=1$};

  % Highlight p(1), p(lambda), and q(1).
  \fill (1.000000,1.414214) circle (0.035) node[above right] {$p(t)$};
  \fill (-0.333333,0.471405) circle (0.035) node[above left] {$p(u)$};
  \draw[->,thick] (0,0) -- (-1.000000,1.414214);
  \fill (-1.000000,1.414214) circle (0.032) node[below left] {$q(t)=p(t)-2e_1$};

  % Origin and center.
  \fill (1,0) circle (0.016) node[below right] {$e$};
  \draw[thick] (0,-0.04)--(0,0.04);
  \draw[thick] (-0.04,0)--(0.04,0);
  \node[below left] at (0,0) {$0$};
\end{tikzpicture}

\caption{The line with normal \(q(t)=p(t)-2e_1\), here with \(t=1\), passes through the consecutive vertices
\(p(t)\) and \(p(u)\) ($m=2$).}
\label{fig:q_t_construction}

\end{figure}
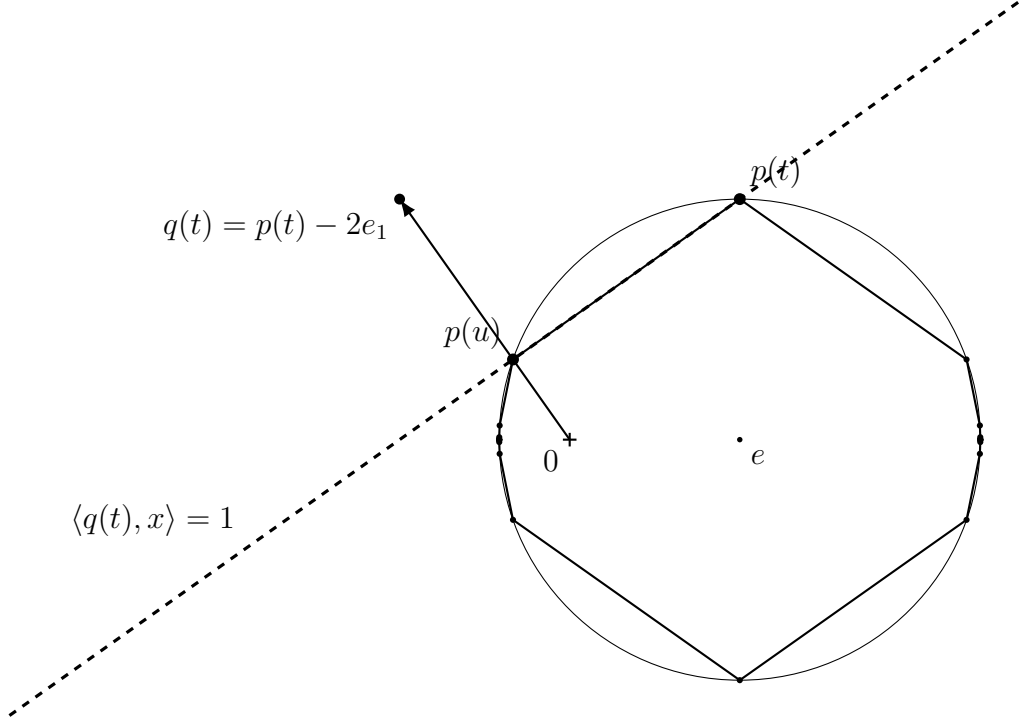

Hence, we require \[
\langle q(t), p(t) \rangle = \langle q(t), p(u) \rangle = 1.\]
Denote
\[
A(t) = \frac{1-t^2}{1+t^2}, \quad B(t) = \frac{2t}{1+t^2}.
\]
Then, since $R^2 -1 = \frac{m^2}{4}$,

\begin{align*}
\langle q(t), p(u) \rangle  
&=
\left(-\frac m2+RA(t)\right)
\left(\frac m2+RA(u)\right)
+R^2B(t)B(u)\\
&=
R^2\bigl(A(t)A(u)+B(t)B(u)\bigr)
+\frac{mR}{2}\bigl(A(t)-A(u)\bigr)
-\frac{m^2}{4} \\
&=
R^2\bigl(A(t)A(u)+B(t)B(u) - 1\bigr)
+\frac{mR}{2}\bigl(A(t)-A(u)\bigr)
+ 1.
\end{align*}

An elementary computation shows
\begin{equation}\label{eq:AAB-identity}
        A(t)A(u)+B(t)B(u)-1
        =
        -\frac{2(t-u)^2}{(1+t^2)(1+u^2)},
\end{equation}
and
\begin{equation}\label{eq:A-difference}
        A(t)-A(u)
        =
        \frac{2(u^2-t^2)}{(1+t^2)(1+u^2)}.
\end{equation}
Thus, 
\begin{align*}
\langle q(t), p(u) \rangle -1
&=
-\frac{2R^2(t-u)^2}{(1+t^2)(1+u^2)}
+
\frac{mR(u^2-t^2)}{(1+t^2)(1+u^2)}\\
&=
\frac{R(t-u)\bigl(-2R(t-u)-m(t+u)\bigr)}{(1+t^2)(1+u^2)}.
\end{align*}
From the above we conclude that the supporting line defined by $\langle q(t), x \rangle = 1$ intersects $\Gamma_s$ at exactly two points: $p(t)$ and $p(u)$ where
\[
u= \frac{2R+m}{2R-m}t = \left(\frac{m + \sqrt{m^2 +4}}{2}\right)^2 t = \lambda t,
\]
where $\lambda =  \left(\frac{m + \sqrt{m^2 +4}}{2}\right)^2$.

We conclude that for every vertex $p(t)$ we must also have a vertex $p(\lambda t)$, so the construction of a polygon is not possible. Hence, we construct an infinite polygon, with vertex set closed under the map $t \mapsto \lambda t$.

To this end, fix some $\alpha > 0$ and consider the set 
\begin{equation}\label{eq:Lambda-alpha}
        \Lambda_\alpha
        :=
        \{0,\infty\}
        \cup
        \{\alpha\lambda^n:n\in\mathbb Z\}
        \cup
        \{-\alpha\lambda^n:n\in\mathbb Z\}.
\end{equation}
The set of positive parameters describes the upper chain of vertices, while the set of negative parameters describes the lower chain.  Define
\begin{equation}\label{eq:K-alpha}
        K_\alpha:=\conv\{p(t):t\in\Lambda_\alpha\}.
\end{equation}
This is an infinite polygon inscribed in $\Gamma_s$.

\begin{figure}[h!]
\centering
\begin{tikzpicture}[scale=2.25,>=Latex]
  % Invariant circle Gamma: center (1,0), radius sqrt(2).
  \draw[thin] (1,0) circle ({sqrt(2)});

  % Finite truncation of K_1.
  \draw[thick]
    (-0.414214,0.000000) --
    (-0.414213,-0.000421) --
    (-0.414211,-0.002451) --
    (-0.414141,-0.014285) --
    (-0.411765,-0.083189) --
    (-0.333333,-0.471405) --
    (1.000000,-1.414214) --
    (2.333333,-0.471405) --
    (2.411765,-0.083189) --
    (2.414141,-0.014285) --
    (2.414211,-0.002451) --
    (2.414213,-0.000421) --
    (2.414214,0.000000) --
    (2.414213,0.000421) --
    (2.414211,0.002451) --
    (2.414141,0.014285) --
    (2.411765,0.083189) --
    (2.333333,0.471405) --
    (1.000000,1.414214) --
    (-0.333333,0.471405) --
    (-0.411765,0.083189) --
    (-0.414141,0.014285) --
    (-0.414211,0.002451) --
    (-0.414213,0.000421) -- cycle;

  \fill (-0.414214,0.000000) circle (0.018);
  \fill (-0.414213,-0.000421) circle (0.018);
  \fill (-0.414211,-0.002451) circle (0.018);
  \fill (-0.414141,-0.014285) circle (0.018);
  \fill (-0.411765,-0.083189) circle (0.018);
  \fill (-0.333333,-0.471405) circle (0.018);
  \fill (1.000000,-1.414214) circle (0.018);
  \fill (2.333333,-0.471405) circle (0.018);
  \fill (2.411765,-0.083189) circle (0.018);
  \fill (2.414141,-0.014285) circle (0.018);
  \fill (2.414211,-0.002451) circle (0.018);
  \fill (2.414213,-0.000421) circle (0.018);
  \fill (2.414214,0.000000) circle (0.018);
  \fill (2.414213,0.000421) circle (0.018);
  \fill (2.414211,0.002451) circle (0.018);
  \fill (2.414141,0.014285) circle (0.018);
  \fill (2.411765,0.083189) circle (0.018);
  \fill (2.333333,0.471405) circle (0.018);
  \fill (1.000000,1.414214) circle (0.018);
  \fill (-0.333333,0.471405) circle (0.018);
  \fill (-0.411765,0.083189) circle (0.018);
  \fill (-0.414141,0.014285) circle (0.018);
  \fill (-0.414211,0.002451) circle (0.018);
  \fill (-0.414213,0.000421) circle (0.018);

  % Center and origin.
  \fill (1,0) circle (0.016) node[below right] {$e$};
  \draw[thick] (0,-0.04)--(0,0.04);
  \draw[thick] (-0.04,0)--(0.04,0);
  \node[below left] at (0,0) {$0$};

  % Endpoint labels.
  \node[right] at (2.414214,0) {$p(0)$};
  \node[left] at (-0.414214,0) {$p(\infty)$};

  \node[above right] at (1.000000,1.414214) {$p(1)$};
  \node[above left] at (-0.333333,0.471405) {$p(\lambda)$};
\end{tikzpicture}
\caption{A finite truncation of the infinite polygon
\(K_1=\operatorname{conv}\{p(t):t\in\Lambda_1\}\) on the circle
\(\Gamma=\{p:\langle p-2e_1,p\rangle=1\}\).}
\label{fig:K_alpha_construction}
\end{figure}
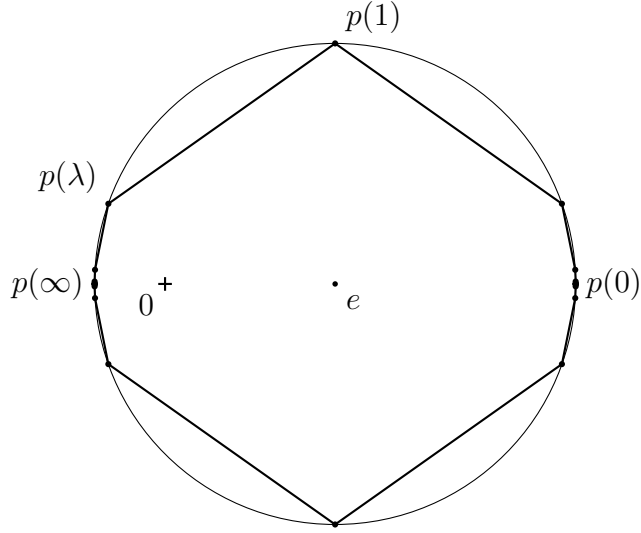

The point $p(0)$ is the right endpoint of the circle and $p(\infty)$ is the left endpoint.  Explicitly,
\[
        p(0)=\left(\frac m2+R,0\right),
        \qquad
        p(\infty)=\left(\frac m2-R,0\right).
\]
Since $R>|m|/2$, these two points lie on opposite sides of the origin.  The set also contains points above and below the $x$-axis. Hence $0\in \intt K_\alpha$ and the polar $K_\alpha^\circ$ is a compact convex body.

We now prove that each shifted vertex $q(t)=p(t)-s$
is indeed the outward polar normal of the corresponding edge of $K_\alpha$.

\begin{lemma}\label{lem:supporting-halfplanes}
For every $t\in\Lambda_\alpha$,
\begin{equation}\label{eq:support-inequality}
        \langle q(t), p(u) \rangle \le 1
        \qquad\forall u\in\Lambda_\alpha.
\end{equation}
If $t\notin\{0,\infty\}$, equality holds at $u=t$ and $u=\lambda t$.  Hence the line
\[
        \langle q(t), x \rangle =1
\]
supports $K_\alpha$ along the edge joining $p(t)$ and $p(\lambda t)$.
\end{lemma}
\begin{proof}
Notice that according to previous computation
\begin{align} \label{eq:q_t_p_u}
\langle q(t), p(u) \rangle -1 = \frac{R(2R-m)(t-u)(u - \lambda t)}{(1+t^2)(1+u^2)}.
\end{align}

First assume $t\notin\{0,\infty\}$.  The two parameters $t$ and $\lambda t$ are consecutive parameters on the corresponding positive or negative orbit.

If $t>0$, then $t<\lambda t$, and there is no parameter $u\in\Lambda_\alpha$ strictly between $t$ and $\lambda t$.  Therefore
\[
        (t-u)(u-\lambda t)\le 0
        \qquad\forall u\in\Lambda_\alpha.
\]
If $t<0$, then $\lambda t<t$, and again no parameter lies strictly between $\lambda t$ and $t$.  The same inequality
\[
        (t-u)(u-\lambda t)\le 0
\]
holds for every selected $u$.

The constant $R(2R - m)$ in (\ref{eq:q_t_p_u}) is positive so
\[
        \langle q(t), p(u)\rangle - 1\le 0.
\]
This proves \eqref{eq:support-inequality}.  Equality holds precisely at $u=t$ and $u=\lambda t$ among the selected parameters.
For $t=0$ and $t=\infty$, the conclusion follows by taking limits.
\end{proof}

We next describe $K_\alpha$ as the intersection of the supporting half-planes found above.

\begin{proposition}\label{prop:halfplane-representation}
For $t \in \Lambda_{\alpha}$ denote the half-space
\[
        H_t=\{x:\ip{q(t)}{x}\le 1\}.
\]
Then, one has
\begin{equation}\label{eq:halfplane-representation}
        K_\alpha
        =
        \bigcap_{t\in\Lambda_\alpha} H_t
\end{equation}
\end{proposition}

\begin{proof}
The inclusion
\[
        K_\alpha
        \subseteq
        \bigcap_{t\in\Lambda_\alpha} H_t\]
follows immediately from Lemma \ref{lem:supporting-halfplanes}, which implies that 
\[
p\left(\Lambda_{\alpha}\right) \subseteq         \bigcap_{t\in\Lambda_\alpha}
        H_t.
\]
Since the right-hand side of the above is convex, we get that
\[
K_{\alpha} = \conv{p\left(\Lambda_{\alpha}\right)} \subseteq \bigcap_{t\in\Lambda_\alpha}H_t.
\]

For the inverse inclusion, let 
\[
        x\in \R^2\setminus K_\alpha.
\]
Since $K_\alpha$ is compact, convex, and contains $0$ in its interior, the strong separation theorem gives a vector $v\ne0$, normalized so that
\begin{equation}\label{eq:normalized-separation}
        \ip{v}{x}>1,
        \qquad
        \ip{v}{y}\le 1\quad\text{for every }y\in K_\alpha.
\end{equation}
Let
\[
        F=K_\alpha\cap\{y:\ip{v}{y}=1\}.
\]
\smallskip
\noindent\emph{Case 1: $F$ is an edge.}
Then
\[
        F=[p(t),p(\lambda t)]
\]
for some $t\in\Lambda_\alpha\setminus\{0,\infty\}$. At a relative interior point of this edge, the tangent cone is a half-plane whose outward normal is $q(t)$. Since the normal cone and the support cone are dual to each other (see \cite{schneider} Section 2.2), we conclude that 
\[
        v = c\cdot q(t).
\]
Both $v$ and $q(t)$ are normalized to equal $1$ on the edge, so $c=1$. From \eqref{eq:normalized-separation},
we have $\ip{q(t)}{x}>1$, so $x\notin H_t$.

\smallskip
\noindent\emph{Case 2: $F$ is a non-accumulation vertex.}
Suppose
\[
        F=\{p(t_0)\},
        \qquad
        t_0\in\Lambda_\alpha\setminus\{0,\infty\}.
\]
The two adjacent edges of $F$ are
\[
        [p(t_0/\lambda),p(t_0)]
        \quad\text{and}\quad
        [p(t_0),p(\lambda t_0)],
\]
with outward normals $q(t_0/\lambda)$ and $q(t_0)$ respectively.
Thus, the normal cone at $p(t_0)$
\[
        N_{K_\alpha}(p(t_0))
        =
      \{\beta_1 q(t_0/\lambda) + \beta_2 q(t_0) \, \vert \, \beta_1, \beta_2 \geq 0\}.
\]
Since $v$ exposes $p(t_0)$,
\[
        v=a q(t_0/\lambda)+b q(t_0)
\]
for some $a,b\ge0$. By definition of $q(t), q(t_0/\lambda)$ we get
\[
        1=\ip{v}{p(t_0)}=a+b.
\]
Hence
\[
        v=(1-b) q(t_0/\lambda)+b q(t_0),
        \qquad
        0 \leq b \leq 1,
\]
Using \eqref{eq:normalized-separation},
\[
        1<\ip{v}{x}
        =(1-b)\ip{q(t_0/\lambda)}{x}+b\ip{q(t_0)}{x}.
\]
so at least one of
\[
        \ip{q(t_0/\lambda)}{x},
        \qquad
        \ip{q(t_0)}{x}
\]
is strictly larger than $1$. Thus $x\notin H_{t_0/\lambda}$ or $x\notin H_{t_0}$.

\smallskip
\noindent\emph{Case 3: $F$ is an accumulation vertex.}
If $F=\{p(0)\}$, then the tangent cone at $p(0)$ is the half-plane
\[
        T_{K_\alpha}(p(0))=\{w:\ip{q(0)}{w}\le0\}.
\]
Indeed, $K_\alpha$ lies in the supporting halfspace $\ip{q(0)}{y}\le1$. Since there exists a sequence of points $\{p(t_n)\}$ in the upper half plane, and respectively $\{p(-t_n)\}$, both on $\Gamma_s$ that converge to $p(0)$, no other halfspace supports $p(0)$. Hence, the normal cone at $p(0)$ is
\[
        N_{K_\alpha}(p(0))=\mathbb R_+q(0).
\]
Normalization gives $v=q(0)$, and therefore \eqref{eq:normalized-separation} implies
\[
        \ip{q(0)}{x}>1.
\]
So $x\notin H_0$.

The case $F=\{p(\infty)\}$ is identical. 
In every case we get,
\[
        x\notin K_\alpha\quad\Longrightarrow\quad x\notin \bigcap_{t\in \Lambda_\alpha} H_t.
\]
so we conclude equality:
\begin{equation}\label{eq:halfspace-representation}
        K_\alpha
        =
        \bigcap_{t\in\Lambda_\alpha}\{x:\ip{q(t)}{x}\le1\}.
\end{equation}

\end{proof}

\begin{proposition}
\[
    K_{\alpha}^\circ = K_{\alpha} - s.
    \]
\end{proposition}

\begin{proof}
    Denote $Y = \{q(t) \vert t\in \Lambda_\alpha\}$. By Proposition \ref{prop:halfplane-representation} we have
    \[
    K_{\alpha} = \{ x : \langle y, x \rangle \leq 1, \forall y \in Y\} = Y^\circ.
    \]
    Thus, 
    \[
    K_{\alpha}^\circ = (Y^\circ)^\circ = \overline{\conv}(Y \cup \{0\}).
    \]
    In our case, $Y$ is compact, and $0 \in \conv(Y)$ so we get \[K_\alpha^\circ = \conv\{q(t) : t\ \in \Lambda_\alpha\} = \conv\{p(t) - s : t\ \in \Lambda_\alpha\} = K_\alpha - s
    \]
\end{proof}

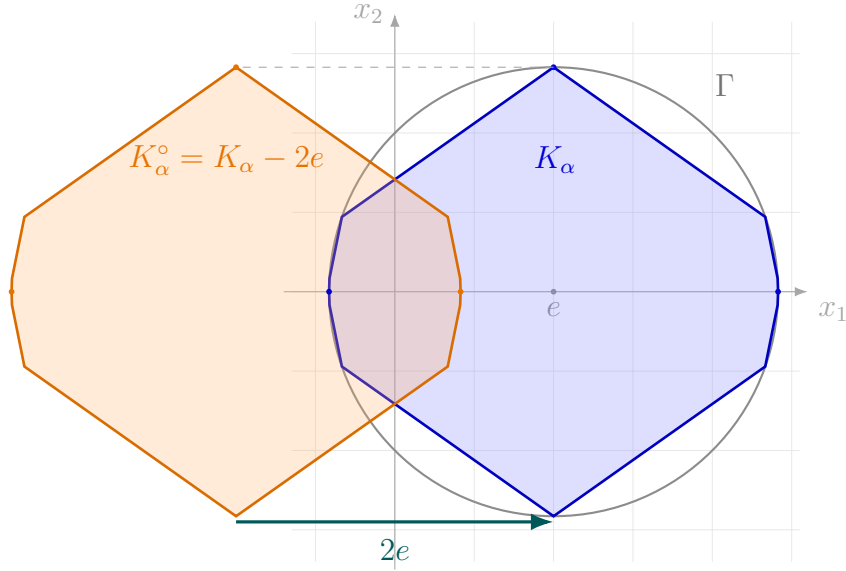
\begin{figure}[h]
\centering
\begin{tikzpicture}[scale=2.1,>=Latex]
  \pgfmathsetmacro{\rtwo}{sqrt(2)}
  \pgfmathsetmacro{\lam}{3+2*sqrt(2)}
  \pgfmathsetmacro{\xleft}{1-sqrt(2)}
  \pgfmathsetmacro{\xright}{1+sqrt(2)}
  \pgfmathsetmacro{\ta}{1/(3+2*sqrt(2))^2}
  \pgfmathsetmacro{\tb}{1/(3+2*sqrt(2))}
  \pgfmathsetmacro{\tc}{1}
  \pgfmathsetmacro{\td}{(3+2*sqrt(2))}
  \pgfmathsetmacro{\te}{(3+2*sqrt(2))^2}

  \newcommand{\pt}[1]{({1 + \rtwo*(1-(#1)^2)/(1+(#1)^2)}, {\rtwo*(2*(#1))/(1+(#1)^2)})}
  \newcommand{\ptshift}[1]{({-1 + \rtwo*(1-(#1)^2)/(1+(#1)^2)}, {\rtwo*(2*(#1))/(1+(#1)^2)})}

  \draw[very thin, gray!18] (-0.65,-1.7) grid[step=.5] (2.55,1.7);
  \draw[->, gray!70] (-0.7,0)--(2.6,0) node[below right] {$x_1$};
  \draw[->, gray!70] (0,-1.75)--(0,1.75) node[left] {$x_2$};

  \draw[thick, black!45] (1,0) circle[radius={sqrt(2)}];
  \fill[black!45] (1,0) circle (0.018) node[below] {$e$};
  \node[black!55] at (2.08,1.30) {$\Gamma$};

  \draw[very thick, -{Latex[length=3mm]}, teal!70!black] (-1.0,-1.45)--(1,-1.45)
    node[midway, below=2pt] {$2e$};

  \filldraw[fill=blue!45, fill opacity=.28, draw=blue!75!black, line width=1pt]
    (\xright,0)
    -- \pt{\ta} -- \pt{\tb} -- \pt{\tc} -- \pt{\td} -- \pt{\te}
    -- (\xleft,0)
    -- \pt{-\te} -- \pt{-\td} -- \pt{-\tc} -- \pt{-\tb} -- \pt{-\ta}
    -- cycle;

  \filldraw[fill=orange!60, fill opacity=.26, draw=orange!85!black, line width=1pt]
    ({\xright-2},0)
    -- \ptshift{\ta} -- \ptshift{\tb} -- \ptshift{\tc} -- \ptshift{\td} -- \ptshift{\te}
    -- ({\xleft-2},0)
    -- \ptshift{-\te} -- \ptshift{-\td} -- \ptshift{-\tc} -- \ptshift{-\tb} -- \ptshift{-\ta}
    -- cycle;

  \fill[blue!80!black] (\xright,0) circle (0.018);
  \fill[blue!80!black] (\xleft,0) circle (0.018);
  \fill[orange!90!black] ({\xright-2},0) circle (0.018);
  \fill[orange!90!black] ({\xleft-2},0) circle (0.018);

  \fill[blue!90!black] \pt{\tc} circle (0.018);
  \fill[orange!90!black] \ptshift{\tc} circle (0.018);
  \draw[dashed, gray!70] \pt{\tc} -- \ptshift{\tc};

  \node[blue!85!black] at (1.01,0.83) {$K_{\alpha}$};
  \node[orange!90!black] at (-1.06,0.83) {$K_{\alpha}^{\circ}=K_{\alpha}-2e$};
\end{tikzpicture}
\caption{ $K_{\alpha}$ (blue) and its translated copy $K_{\alpha}-2e=K_{\alpha}^{\circ}$ (orange). $m=2.$}
\label{fig:K_alpha_and_dual}
\end{figure}

\subsection{Proof of Theorem \ref{thm-main1} in the plane}
To complete our proof, let \(s=(m,0)\), \(m>0\), and set
     \[
     \lambda=\left(\frac{m+\sqrt{m^2+4}}2\right)^2.
     \]
Define $K = K_1$ and $T=K_{\sqrt{\lambda}}$. \\
Let us verify that $T \neq K$. Indeed, $p(1)$ is a vertex of $K$ and the two nearest vertices of $T$ are $p(\lambda^{-1/2})$ and $p(\lambda^{1/2})$. Notice that 
\[
\lambda^{-1/2} < 1 < \lambda^{1/2},
\]
so by equality (\ref{eq:q_t_p_u})
it must hold that $\langle q(\lambda^{-1/2}), p(1) \rangle > 1$ so $p(1) \notin K_{\sqrt{\lambda}}$, which implies $K\neq T.$ For a general  $0\neq s\in \R^2$, one may rotate the construction so that $s$ becomes $(|s|,0)$.
\subsection{The general construction}\label{sec:revolution}

The planar construction extends to every dimension by revolving the infinite polygon around the translation axis.  

Let \(L\subset\R^2\) be a compact convex body that is symmetric with respect to the
first coordinate axis, that is,
\[
        (x_1,x_2)\in L\quad\Longleftrightarrow\quad (x_1,-x_2)\in L.
\]
For \(n\ge2\), write \(\R^n=\R\oplus\R^{n-1}\), and define the body of revolution
of \(L\) by
\begin{equation}\label{eq:revolution-definition}
        \mathcal R_n(L)
        :=\left\{(x_1,x')\in\R\oplus\R^{n-1}:
        (x_1,\|x'\|)\in L\right\}.
\end{equation}

\begin{lemma}\label{prop:polar-revolution}
Let \(L\subset\R^2\) be a compact convex body, symmetric with respect to the first
coordinate axis, and suppose \(0\in\intt L\).  Then \(\mathcal R_n(L)\) is a
compact convex body with \(0\in\intt\mathcal R_n(L)\), and
\begin{equation}\label{eq:polar-revolution}
        \mathcal R_n(L)^\circ=\mathcal R_n(L^\circ),
\end{equation}
where the polar on the right is taken in \(\R^2\).
\end{lemma}

\begin{proof}
First, notice that since polarity commutes with orthogonal maps, $\mathcal R_n(L)^\circ$ must be invariant under all rotations fixing $e_1$. This implies that $\mathcal R_n(L)^\circ$ is also a body of revolution around $e_1$.
Fix a unit vector
\(u \in e_1^\perp\) and define
\[
        E_u:=\operatorname{span}\{e_1,u\}.
\]
Clearly $P_{E_u} \mathcal R_n(L) = \mathcal R_n(L)\cap E_u = L$ (in the corresponding subspace). Using the section-projection correspondence of duality we get
\[
\mathcal R_n(L)^\circ \cap E_u = \left(P_{E_u} \mathcal R_n(L)\right)^\circ = L^\circ.
\]
Since $\left(\mathcal R_n(L)\right)^\circ$ is a body of revolution we conclude \(\mathcal R_n(L)^\circ=\mathcal R_n(L^\circ)\), as claimed.
\end{proof}

Since \(K_\alpha^\circ=K_\alpha-s\) in the plane, and since translation along
     the
 \(e_1\)-axis commutes with revolution, we get
 \[
 \mathcal R_n(K_\alpha)^\circ
 =
 \mathcal R_n(K_\alpha^\circ)
 =
 \mathcal R_n(K_\alpha-s)
 =
 \mathcal R_n(K_\alpha)-s.
 \]
 Thus \(\mathcal R_n(K_1)\) and \(\mathcal R_n(K_{\sqrt\lambda})\) satisfy the
 required identities. They are distinct because their intersections with the
plane \(\operatorname{span}\{e_1,e_2\}\) are the distinct planar
     bodies
 \(K_1\) and \(K_{\sqrt\lambda}\).

\section{Proof of Theorem \ref{thm-main2}}
To construct our sets, we use a ball bounded by $\Gamma_s$ in the proof of Theorem \ref{thm-main1}. That is,
\[
K = \{x \in \R^n : \ip{x}{x-a} \leq 1\}, \qquad 0\neq a \in \R^n.
\]
We claim that $K - a = (K^\circ + a)^\circ$. Indeed, having $x\in (K^\circ + a)^\circ$ is equivalent to $\ip{x}{y+a} \leq 1$ for all $y\in K^\circ$, which can be reformulated as $h_{K^\circ}(x) + \ip{x}{a} \leq 1$. Since $h_{K^\circ}(\cdot) = p_K(\cdot)$, where $p_K$ is Minkowski gauge functional of $K$, we get
\[
x\in (K^\circ + a)^\circ \iff p_K(x) + \ip{x}{a} \leq 1.
\]
By definition: 
\[
p_K(x) = \inf\{\lambda > 0 : x \in \lambda K\} = \inf\left\{\lambda > 0 : \ip{\frac{x}{\lambda}}{\frac{x}{\lambda} - a} \leq 1\right\}.
\]
Clearly, the infimum is achieved at 
\[
\lambda = \frac{-\ip{a}{x}+\sqrt{\ip{a}{x}^2 + 4|x|^2}}{2}.
\]
Hence
\[
p_K(x) + \ip{x}{a} \leq 1 \iff \frac{\ip{a}{x}+\sqrt{\ip{a}{x}^2 + 4|x|^2}}{2} \leq 1,
\]
which is equivalent to the condition
\[
|x|^2 + \ip{a}{x} \leq 1,
\]
or $\ip{x}{x+a} \leq 1$, which means $x \in K-a$.

Now, set $T = K^\circ + a$. Clearly, $T + T^\circ = K^\circ + a + K - a = K^\circ + K$.
Note that $T\neq K$ since $K^\circ$ is a proper ellipsoid, and clearly $T \neq K^\circ$.

\bibliographystyle{plain} % Or whichever style you prefer
\bibliography{references} % The name of your .bib file (without the .bib extension)

@book{schneider,
title={Convex bodies: the Brunn--Minkowski theory},
  author={Schneider, Rolf},
  volume={151},
  year={2013},
  publisher={Cambridge university press}
}

@article{AM-duality,
 title={A characterization of the concept of duality},
  author={Artstein-Avidan, Shiri and Milman, Vitali},
  journal={Electronic Research Announcements in Mathematical Sciences},
  volume={14},
  pages={42--59},
  year={2007}
}

@article{jensen,
  title={Self-polar polytopes},
  author={Jensen, Alathea},
  journal={Polytopes and Discrete Geometry},
  year={2019}
}

@article{AM-legendre,
  title={The concept of duality in convex analysis, and the characterization of the Legendre transform},
  author={Artstein-Avidan, Shiri and Milman, Vitali},
  journal={Annals of mathematics},
  pages={661--674},
  year={2009},
  publisher={JSTOR}
}

@article{BS-duality,
 title={A characterization of the duality mapping for convex bodies},
  author={B{\"o}r{\"o}czky, K{\'a}roly J and Schneider, Rolf},
  journal={Geometric and Functional Analysis},
  volume={18},
  number={3},
  pages={657--667},
  year={2008},
  publisher={Springer}
}

@article{gruber,
  title={The endomorphisms of the lattice of convex bodies},
  author={Gruber, Peter M},
  booktitle={Abhandlungen aus dem Mathematischen Seminar der Universit{\"a}t Hamburg},
  volume={61},
  number={1},
  pages={121--130},
  year={1991},
  organization={Springer}
}

@article{slomka,
  title={On duality and endomorphisms of lattices of closed convex sets.},
  author={Slomka, Boaz A},
  journal={Advances in Geometry},
  volume={11},
  number={2},
  year={2011}
}

\end{document}